\documentclass[11pt]{amsart}

\usepackage{amsmath}
\usepackage{amssymb}
\usepackage{amsfonts}
\usepackage{amscd}
\usepackage[enableskew]{youngtab}

\usepackage[usenames,dvipsnames]{pstricks}
\usepackage{epsfig}
\usepackage{pst-grad} 
\usepackage{pst-plot} 
\usepackage{caption}

\newcommand{\nc}{\newcommand}

\nc{\exto}[1]{\stackrel{#1}{\longrightarrow}}
\nc{\dlim}{{\mathop{\lim\limits_{\longrightarrow}\,}}}
\nc{\ilim}{{\mathop{\lim\limits_{\longleftarrow}\,}}}
\nc{\hocolim}{{\mathop{\sf hocolim}\,}}
\nc{\holim}{{\mathop{\sf holim}}}
\nc{\lan}{\big\langle}
\nc{\ran}{\big\rangle}

\nc{\kk}{{\mathsf{k}}}

\nc{\C}{{\mathbb{C}}}
\nc{\HH}{{\mathbf{H}}}
\nc{\DD}{{\mathbb{D}}}
\nc{\LL}{{\mathbb{L}}}
\nc{\PP}{{\mathbb{P}}}
\nc{\QQ}{{\mathbb{Q}}}
\nc{\RR}{{\mathbb{R}}}
\nc{\ZZ}{{\mathbb{Z}}}

\nc{\CA}{{\mathcal{A}}}
\nc{\CB}{{\mathcal{B}}}
\nc{\CC}{{\mathcal{C}}}
\nc{\D}{{\mathcal{D}}}
\nc{\CE}{{\mathcal{E}}}
\nc{\CF}{{\mathcal{F}}}
\nc{\CG}{{\mathcal{G}}}
\nc{\CH}{{\mathcal{H}}}
\nc{\CL}{{\mathcal{L}}}
\nc{\CM}{{\mathcal{M}}}
\nc{\CN}{{\mathcal{N}}}
\nc{\CO}{{\mathcal{O}}}
\nc{\CQ}{{\mathcal{Q}}}
\nc{\CR}{{\mathcal{R}}}
\nc{\CS}{{\mathcal{S}}}
\nc{\CT}{{\mathcal{T}}}
\nc{\CU}{{\mathcal{U}}}
\nc{\CV}{{\mathcal{V}}}
\nc{\CW}{{\mathcal{W}}}
\nc{\CX}{{\mathcal{X}}}
\nc{\CY}{{\mathcal{Y}}}
\nc{\CMo}{{\mathcal{M}^\circ}}
\nc{\Co}{{{C}^\circ}}

\nc{\BY}{{\overline{Y}}}
\nc{\BYD}{{\overline{Y}{}^{|D|}}}
\nc{\OZ}{{\overline{Z}}}
\nc{\bg}{{\bar{g}}}

\nc{\bq}{{\mathbf{q}}}
\nc{\BD}{{\mathbf{D}}}
\nc{\BG}{{\mathbf{G}}}
\nc{\BM}{{\mathbf{M}}}
\nc{\BP}{{\mathbf{P}}}
\nc{\BZ}{{\mathbf{Z}}}
\nc{\BPr}{{\mathsf{P}}}
\nc{\BL}{{\mathbf{L}}}
\nc{\BR}{{\mathbf{R}}}
\nc{\BRO}[1]{{{\mathbf{R}}^{\circ}_{#1}}}
\nc{\BRD}[1]{{{\mathbf{R}}^{|D|}_{#1}}}
\nc{\BRP}[1]{{{\mathbf{R}}^{1}_{#1}}}
\nc{\BRTP}[1]{{{\mathbf{\tilde{R}}}{}^{1}_{#1}}}
\nc{\BS}{{\mathbf{S}}}
\nc{\BMS}{{{\mathbf{M}}^{{s}}}}
\nc{\BMSS}{{{\mathbf{M}}^{{ss}}}}
\nc{\BMZ}{{\mathbf{M}^{\circ}}}
\nc{\BCL}{{\mathbf{L}}}

\nc{\PCC}{{{}^\perp\CC}}

\nc{\Cl}{{\mathsf{Cliff}}}
\nc{\Clev}{{\mathop{\mathsf{Cliff}}^{\circ}}}

\nc{\FA}{{\mathfrak{A}}}
\nc{\FB}{{\mathfrak{B}}}

\nc{\fa}{{\mathfrak{a}}}
\nc{\fb}{{\mathfrak{b}}}
\nc{\fg}{{\mathfrak{g}}}
\nc{\fn}{{\mathfrak{n}}}
\nc{\fp}{{\mathfrak{p}}}
\nc{\FD}{{\mathfrak{D}}}
\nc{\FE}{{\mathfrak{E}}}
\nc{\FL}{{\mathfrak{L}}}
\nc{\FM}{{\mathfrak{M}}}
\nc{\FS}{{\mathsf{S}}}

\nc{\sfc}{{\mathsf{c}}}
\nc{\sfch}{{\mathsf{ch}}}
\nc{\sfh}{{\mathsf{h}}}

\nc{\SK}{{\mathsf{K}}}
\nc{\SM}{{\mathsf{M}}}
\nc{\SO}{{\mathsf{O}}}
\nc{\SQ}{{\mathsf{Q}}}
\nc{\SPV}{{\mathsf{S}^+\mathsf{V}}}
\nc{\SMV}{{\mathsf{S}^-\mathsf{V}}}
\nc{\SPMV}{{\mathsf{S}^\pm\mathsf{V}}}
\nc{\SX}{{S_X}}
\nc{\SY}{{S_Y}}
\nc{\phipsi}{{q}}
\nc{\eps}{\varepsilon}

\nc{\pim}{{\pi_-}}
\nc{\pip}{{\pi_+}}

\nc{\BE}{{\overline{\CE}}}
\nc{\TE}{{\tilde{\CE}}}
\nc{\TQ}{{\tilde{Q}}}
\nc{\TCF}{{\tilde{\CF}}}
\nc{\TCG}{{\tilde{\CG}}}
\nc{\TCL}{{\tilde{\CL}}}
\nc{\TF}{{\tilde{F}}}
\nc{\TW}{{\tilde{W}}}
\nc{\TCC}{{\tilde{\CC}}}
\nc{\TCX}{{\tilde{\CX}}}
\nc{\TCY}{{\tilde{\CY}}}
\nc{\TPhi}{{\tilde{\Phi}}}
\nc{\OPhi}{{\bar{\Phi}}}
\nc{\txi}{{\tilde{\xi}}}
\nc{\tp}{{\tilde{p}}}
\nc{\tq}{{\tilde{q}}}
\nc{\tzeta}{{\tilde{\zeta}}}
\nc{\tpi}{{\tilde{\pi}}}

\nc{\halpha}{{\hat{\alpha}}}
\nc{\HCA}{{\hat{\CA}}}
\nc{\HCB}{{\hat{\CB}}}
\nc{\HCC}{{\hat{\CC}}}
\nc{\HE}{{\widehat{\CE}}}
\nc{\HX}{{\hat{X}}}
\nc{\hxi}{{\hat{\xi}}}

\nc{\UH}{{\mathcal{H}}}

\nc{\TM}{{\widetilde{M}}}
\nc{\TCM}{{\widetilde{\CM}}}
\nc{\TU}{{\widetilde{U}}}
\nc{\TX}{{\widetilde{X}}}
\nc{\TY}{{\widetilde{Y}}}
\nc{\TYO}{{{\widetilde{Y}}^\circ}}
\nc{\barf}{{\bar{f}}}
\nc{\te}{{\tilde{e}}{}}
\nc{\tf}{{\tilde{f}}}
\nc{\tg}{{\tilde{g}}}
\nc{\ti}{{\tilde{\imath}}}
\nc{\tj}{{\tilde{\jmath}}}
\nc{\ty}{{\tilde{y}}}
\nc{\tphi}{{\tilde{\phi}}}

\nc{\urho}{{\underline{\rho}}}

\nc{\LRA}{\Leftrightarrow}
\nc{\RA}{\Rightarrow}
\nc{\lotimes}{\mathbin{\mathop{\otimes}\limits^{\mathbb{L}}}}
\nc{\CEnd}{\mathop{\mathcal{E}\mathit{nd}}\nolimits}
\nc{\CExt}{\mathop{\mathcal{E}\mathit{xt}}\nolimits}
\nc{\CHom}{\mathop{\mathcal{H}\mathit{om}}\nolimits}
\nc{\RH}{\mathop{{\mathsf{R}}\Gamma}\nolimits}
\nc{\RGamma}{\mathop{{\mathsf{R}}\Gamma}\nolimits}
\nc{\RHom}{\mathop{\mathsf{RHom}}\nolimits}
\nc{\RCHom}{\mathop{\mathsf{R}\mathcal{H}\mathit{om}}\nolimits}
\nc{\RG}{\mathop{\mathsf{R\Gamma}}\nolimits}
\nc{\Hom}{\mathop{\mathsf{Hom}}\nolimits}
\nc{\Ext}{\mathop{\mathsf{Ext}}\nolimits}
\nc{\End}{\mathop{\mathsf{End}}\nolimits}
\nc{\Tor}{\mathop{\mathsf{Tor}}\nolimits}
\nc{\Tordim}{\mathop{\mathsf{Tor}\text{\rm-}\mathsf{dim}}\nolimits}
\nc{\Hilb}{\mathop{\mathsf{Hilb}}\nolimits}
\nc{\Spec}{\mathop{\mathsf{Spec}}\nolimits}
\nc{\Pic}{\mathop{\mathsf{Pic}}\nolimits}

\nc{\Tr}{\mathop{\mathsf{Tr}}\nolimits}
\nc{\Cone}{\mathop{\mathsf{Cone}}\nolimits}
\nc{\Fiber}{\mathop{\mathsf{Fiber}}\nolimits}
\nc{\Ker}{\mathop{\mathsf{Ker}}\nolimits}
\nc{\Coker}{\mathop{\mathsf{Coker}}\nolimits}
\nc{\codim}{\mathop{\mathsf{codim}}\nolimits}
\nc{\sing}{{\mathsf{sing}}}
\nc{\supp}{\mathop{\mathsf{supp}}}
\nc{\vol}{\mathop{\mathsf{vol}}\nolimits}
\nc{\ch}{\mathop{\mathsf{ch}}\nolimits}
\nc{\perf}{{\mathsf{perf}}}
\nc{\rank}{\mathop{\mathsf{rank}}}
\nc{\Pf}{{\mathsf{Pf}}}
\nc{\Gr}{{\mathsf{Gr}}}
\nc{\OGr}{{\mathsf{OGr}}}
\nc{\Flag}{{\mathsf{Fl}}}
\nc{\Kosz}{{\mathsf{Kosz}}}
\nc{\LGr}{{\mathsf{LGr}}}
\nc{\GTGr}{{\mathsf{G_2Gr}}}
\nc{\GTF}{{\mathsf{G_2F}}}
\nc{\OF}{{\mathsf{OF}}}
\nc{\Fl}{{\mathsf{Fl}}}
\nc{\Bl}{{\mathsf{Bl}}}
\nc{\GL}{{\mathsf{GL}}}
\nc{\PGL}{{\mathsf{PGL}}}
\nc{\SL}{{\mathsf{SL}}}
\nc{\SP}{{\mathsf{Sp}}}
\nc{\Spin}{{\mathsf{Spin}}}
\nc{\Tot}{{\mathsf{Tot}}}
\nc{\ev}{{\mathsf{ev}}}
\nc{\od}{{\mathsf{odd}}}
\nc{\coev}{{\mathsf{coev}}}
\nc{\id}{{\mathsf{id}}}
\nc{\opp}{{\mathsf{opp}}}
\nc{\PS}{{{\PP^3}}}
\nc{\Qu}{{{Q^3}}}
\nc{\tdim}{\mathop{\Tor\dim}}
\nc{\ecart}{{\fbox{$\scriptstyle\mathsf{EC}$}}}
\nc{\ad}{{\mathop{\mathsf ad}}}
\nc{\sg}{{\mathop{\mathsf sg}}}
\nc{\hf}{{\mathop{\mathsf hf}}}
\nc{\gr}{{\mathop{\mathsf gr}}}
\nc{\qgr}{{\mathop{\mathsf qgr}}}
\nc{\Coh}{{\mathop{\mathsf Coh}}}
\nc{\Ab}{{\mathop{\mathcal{A}\mathit{b}}}}
\nc{\Ccoh}{{\mathop{\mathsf Ccoh}}}
\nc{\Qcoh}{{\mathop{\mathsf Qcoh}}}
\nc{\At}{{\mathop{\mathsf{At}}\nolimits}}
\nc{\tra}{{\mathsf{T}}}
\nc{\fsl}{{\mathfrak{sl}}}
\nc{\fso}{{\mathfrak{so}}}
\nc{\fgl}{{\mathfrak{gl}}}

\nc{\AAV}{{\mathcal{AAV}}}

\nc{\Rep}{{\mathsf{Rep}}}

\nc{\Cubics}{{{\mathcal{S}}_3}}
\nc{\VFT}{{{\mathcal{S}}_{14}}}
\nc{\VFTE}{{{\mathcal{N}}_{\mathrm{reg,sm}}}}
\nc{\MX}{{\CM_X}}
\nc{\MY}{{\CM_Y}}
\nc{\MYE}{{\CM_{Y,\CE}}}
\nc{\Yd}{{Y_d}}
\nc{\Yfive}{{Y_5}}
\nc{\Xg}{{X_{2g-2}}}
\nc{\Xtt}{{X_{22}}}
\nc{\Xst}{{X_{16}}}
\nc{\Xtw}{{X_{12}}}
\nc{\Xe}{{X_{8}}}
\nc{\Xf}{{X_{4}}}

\nc{\git}{{/\!\!/\!{}_\chi}}

\nc{\HOH}{{\mathsf H\mathsf H}}
\nc{\HHE}{{\mathsf H\mathsf E}}

\nc{\You}{{\mathsf{Y}}}
\nc{\Youu}{\mathsf{Y}^\mathrm{u}}
\nc{\Youmu}{\mathsf{Y}^\mathrm{mu}}
\nc{\Youl}{\mathsf{Y}^\mathrm{l}}
\nc{\Youml}{\mathsf{Y}^\mathrm{ml}}
\nc{\CBp}{\CB^\prime}
\nc{\lort}[1]{\vphantom{#1}^\perp{}#1}

\theoremstyle{plain}

\newtheorem{theorem}{Theorem}[section]
\newtheorem{conjecture}[theorem]{Conjecture}
\newtheorem{lemma}[theorem]{Lemma}
\newtheorem{proposition}[theorem]{Proposition}
\newtheorem{corollary}[theorem]{Corollary}

\theoremstyle{definition}

\newtheorem{definition}[theorem]{Definition}

\theoremstyle{remark}

\newtheorem{remark}[theorem]{Remark}
\newtheorem{example}[theorem]{Example}

\title{On minimal Lefschetz decompositions for Grassmannians}
\author{Anton Fonarev}
\address{
	\sloppy\parbox{0.9\textwidth}{
		Algebra Section, Steklov Mathematical Institute,
		8 Gubkina str., Moscow 119991 Russia
		\hfill\\[1pt]
		Laboratory of Algebraic Geometry, SU-HSE,
		7 Vavilova Str., Moscow 117312 Russia
		\hfill
	}\bigskip
}
\email{avfonarev@mi.ras.ru}
\date{}
\dedicatory{Dedicated to the blessed memory of my Grandfathers.}
\thanks{The author was partially supported by AG Laboratory NRU-HSE, RF government grant, ag. 11.G34.31.0023, RFFI grants 11-01-92613-KO-a, 10-01-00678-a, NSh-5139.2012.1 and by the Moebius Contest Foundation for Young Scientists.}

\begin{document}

\begin{abstract}
	We construct two Lefschetz decompositions of the derived category of coherent sheaves on the Grassmannian of $k$-dimensional subspaces in a vector space of
	dimension $n$. Both of them admit a Lefschetz basis consisting of equivariant vector bundles. We prove fullness of the first decomposition and conjecture it
	for the second one. In the case when $n$ and $k$ are coprime these decompositions coincide and are minimal. In general, we conjecture minimality of
	the second decomposition.
\end{abstract}

\maketitle

\section{Introduction} 
\label{sec:introduction}

The derived category of coherent sheaves is one of the most important invariants of an algebraic variety. In order to study derived categories one may look for
additional structure. There is a particular case when a triangulated category (and any derived category is triangulated) can be described explicitly.

\begin{definition}[\cite{BondalRepr}, \cite{GorRudProj}]
	A collection of objects $(E_1,E_2,\ldots,E_n)$ in a $\kk$-linear triangulated category $\CT$ is called \textbf{exceptional} if
	\[
		\RHom(E_i,E_i)=\kk\quad\text{for all } i,\qquad \RHom(E_i,E_j)=0\quad\text{for }i>j.
	\]
	An exceptional collection $(E_1,E_2,\ldots,E_n)$ is called \textbf{full} if $\CT$ is the smallest full triangulated subcategory of $\CT$ containing all $E_i$.
\end{definition}

Exceptional collections can be considered as a kind of basis for the triangulated category: in this situation every object admits a unique filtration with the $i$-th
quotient isomorphic to a direct sum of shifts of the $i$-th object in the collection.

The simplest example of a variety with a full exceptional collection is a projective space. A.\,Beilinson showed in~\cite{Beilinson} that the collection
$\left(\CO_{\PP^n},\CO_{\PP^n}(1),\ldots\CO_{\PP^n}(n)\right)$ is full and exceptional in $\D^b(\PP^n)$.

One can slightly generalize the notion of an exceptional collection and consider  subcategories instead of single objects. For a full triangulated subcategory
$\CA$ of a triangulated category $\CT$ the \textbf{right orthogonal} (resp. \textbf{left orthogonal}) to $\CA$ in $\CT$ is the full triangulated subcategory
$\CA^\perp$ (resp. $\lort{\CA}$) consisting of all the objects $T\in\CT$ such that $\Hom_{\CT}(A,T) = 0$ (resp. $\Hom_{\CT}(T, A)=0$) for all $A\in\CA$.

Given a sequence of full triangulated subcategories $\CA_1,\CA_2,\ldots,\CA_n\subset \CT$ we denote by $\left<\CA_1,\CA_2,\ldots,\CA_n\right>$ the smallest full triangulated
subcategory of $\CT$ containing $\CA_1,\CA_2,\ldots,\CA_n$.

\begin{definition}[\cite{BondalKapranov}]
	A sequence $\CA_1,\CA_2,\ldots,\CA_n$ of full triangulated subcategories in a triangulated category $\CT$ is called a \textbf{semi-orthogonal collection} if
	$\CA_i\subset\lort{\CA_j}$ for all $i>j$. A semi-orthogonal collection $\CA_1,\CA_2,\ldots,\CA_n$ is called a \textbf{semi-orthogonal decomposition} if
	$\CT=\left<\CA_1,\CA_2,\ldots,\CA_n\right>$.
\end{definition}

There is a very interesting class of semi-orthogonal decompositions in the case of derived categories. These are called Lefschetz and are of the form
\[
	\D^b(X)=\left<\CB_0,\CB_1\otimes L,\ldots,\CB_{m-1}\otimes L^{m-1}\right>,
\]
where $\D^b(X)\supset \CB_0\supset\ldots\supset \CB_{m-1}$ are full triangulated subcategories and $L$ is a fixed line bundle.

It turns out that every Lefschetz decomposition is completely determined by its first block $\CB_0$. Therefore, there is a natural partial order on the set of
all Lefschetz decompositions induced by the inclusion order on the set of their first blocks. An important problem is to
construct Lefschetz decompositions that are minimal with respect to this partial order.

Minimal Lefschetz decompositions are interesting due to several reasons. One of them is the following conjecture by A.\,Kuznetsov.

\begin{conjecture}[\cite{KuznResol}]
	Let $X$ be a smooth projective variety. Then minimal Lefschetz decompositions of $\D^b(X)$ correspond to minimal categorical resolutions of singularities of
	the affine cone over~$X$.
\end{conjecture}

In the present paper we construct two Lefschetz decompositions of the bounded derived category of coherent sheaves on a Grassmannian $X=\Gr(k, V)$ of
$k$-dimensional subspaces in a vector space $V$ of dimension $n$.

In~\cite{KapranovGr} M.\,Kapranov constructed a full exceptional collection in the derived category of $X$
\[
	\D^b(X)=\left<\Sigma^\lambda\CU\mid\lambda\in\You_{n,k}\right>,
\]
where $\lambda$ runs over the set $\You_{n,k}$ of Young diagrams inscribed in a rectangle of size $k\times(n-k)$ and $\CU$ denotes the tautological subbundle of
rank $k$ in $V\otimes\CO_X$. However, Kapranov's collection gives rise to a highly non-minimal Lefschetz decomposition.

We consider two Lefschetz decompositions of $\D^b(X)$. The first one was independently discovered by C.\,Brav and H.\,Thomas~(\cite{Brav}). However, they
were only able to prove semi-orthogonality of this decomposition, but not fullness, and have never published their result. In the present paper we prove
both. In the case when $n$ and $k$ are coprime $(n,k)=1$ this Lefschetz decomposition is minimal. The first block of this decomposition is
\[
	\CB_0=\left<\Sigma^\lambda\CU^*\mid\lambda\in\Youu_{n,k}\right>,
\]
where $\Youu_{n,k}\subset\You_{n,k}$ denotes the set of those diagrams that do not go below the diagonal going from the lower left to the upper right corner.

It turns out that whenever $n$ and $k$ are not coprime, we can construct a smaller Lefschetz decomposition. However, at the moment we are
only able to prove semi-orthogonality of this decomposition. This second Lefschetz decomposition is conjectured to be full and minimal.

Finally, we should mention that in the case $X=\Gr(2,V)$ a minimal Lefschetz decomposition was constructed in~\cite{KuznExc}. It coincides with both
decompositions described in the present paper.

The paper is organized as follows. In Section~\ref{sec:preliminaries} we recall all the definitions and facts we need about Lefschetz decompositions, the
celebrated Borel--Bott--Weil theorem and Littlewood--Richardson rule and Kapranov's exceptional collections for Grassmannians.
In Section~\ref{sec:combinatorics_of_young_diagrams} we introduce a $\ZZ/n\ZZ$-group action on the set $\You_{n,k}$ and some characteristics related
to it. This action will play a significant role in the construction of our decompositions and in the proof of fullness.
In Section~\ref{sec:lefschetz_decompositions_for_grassmannians} we state our main results and conjectures and prove semi-orthogonality of the two Lefschetz
decompositions.
Finally, in Section~\ref{sec:fullness} we construct a new and highly interesting class of exact complexes and use them to prove fullness of the first decomposition.

\subsection*{Acknowledgements} 
I am extremely grateful to my advisor S.\,M.\,Gusein-Zade for his care, constant attention and mathematical tolerance. This work could not have been done without
A.\,Kuznetsov, who not only proposed this problem to me, but basically taught me his beautiful vision of algebraic geometry. I thank him, A.\,Bondal,
L.\,Manivel and D.\,Orlov for helpful discussions. Finally, I thank F.\,El Zein for inviting me to ICTP in summer 2010, a unique and wonderful place where this work was started.

The last, but not the least, I am thankful to my family for their support and to Caroline for being a great inspiration.



\section{Preliminaries} 
\label{sec:preliminaries}

\subsection{Lefschetz decompositions} 
\label{sub:lefschetz_decompositions}

We will be interested in a special class of semi-orthogonal decompositions of the bounded derived category of coherent sheaves on an algebraic variety. Let $X$ be
an algebraic variety over a field $\kk$ of characteristic zero and $\CO_X(1)$ a line bundle on $X$.

\begin{definition}[\cite{KuznDual}]
	A \textbf{Lefschetz decomposition} of $\D^b(X)$ is a semi-orthogonal decomposition of the form
\[
	\D^b(X)=\left<\CB_0,\CB_1(1),\ldots,\CB_{m-1}(m-1)\right>,\qquad \text{where } 0\subset\CB_{m-1}\subset \ldots \subset \CB_1 \subset \CB_0 \subset \D^b(X).
\]
	The category $\CB_i\subset\D^b(X)$ is called the $(i+1)$-th \textbf{block} of the decomposition.
\end{definition}

\begin{definition}
	Let $(E_1, E_2, \ldots, E_n)$ be an exceptional collection and $o:\{1,\dots,n\}\to\ZZ_{>0}$ be a positive integer-valued function, such that the
	categories
\begin{equation}\label{eq:basis}
	\CB_i=\left<E_j\mid i<o(j)\right>
\end{equation}
	form a Lefchetz decomposition. In this case we say that the collection $(E_1, E_2, \ldots, E_n)$ is a \textbf{Lefschetz basis} of 
	$\D^b(X)$ with the \textbf{support function}~$o$.
\end{definition}

\begin{remark}
	The reader might have an impression that the support function comes as a part of data in the definition of a Lefschetz basis. However, every Lefschetz
	decomposition is completely determined by its first block (see~\cite{KuznResol}) by the following inductive rule:
\[
	\CB_k= \vphantom{\CB_0(-k)}^\perp\CB_0(-k)\cap\CB_{k-1}.
\]
	Thus, one can say that a Lefschetz basis is an exceptional collection $(E_1, E_2, \ldots, E_n)$ in $\D^b(X)$, such that
	$\CB_0=\left<E_1,E_2,\ldots,E_n\right>$ is the first block of a Lefschetz decomposition and every other block $\CB_i$ is generated by a subcollection of
	$(E_1,E_2,\ldots,E_n)$.
\end{remark}
A Lefschetz decomposition is called \textbf{rectangular} if $\CB_{m-1}=\ldots=\CB_1=\CB_0$, and \textbf{minimal}, if it is minimal with
respect to the partial ordering given by the inclusion of the first block.

\begin{remark}
	The notion of a Lefschetz basis is closely related to the notion of a full Lefschetz exceptional collection (see~\cite{KuznExc}). Whenever we have a Lefschetz basis with a
	nonincreasing support function, we get a full Lefschetz exceptional collection, and vice versa.
	
	A reader familiar with mutations of exceptional collections may note that mutations are may defined for Lefschetz bases as well. Moreover, one can always mutate a Lefschetz basis so that it will correspond to a Lefschetz
	exceptional collection. It's enough to mutate objects with smaller support function value to the right within the first block.
\end{remark}

We will need the following simple lemma that gives an easy way to check semi-orthogonality of the blocks generated by a Lefschetz basis.

\begin{lemma}\label{lemma:basis}
	An exceptional collection $(E_1,E_2,\ldots,E_n)$ is a Lefschetz basis with the support function $o(i)$ if and only if
	\begin{enumerate}
		\item the subcategories $\CB_i(i)$ defined by~\eqref{eq:basis} generate $\D^b(X)$, and
		\item $\Ext^\bullet(E_p(k), E_q)=0$ for all $1\leq p,q\leq n$ and $0< k<o(p)$.
	\end{enumerate}
\end{lemma}
\begin{proof}
	The first condition is just the fullness of the generated decomposition. As for the second one, it is sufficient to note that the blocks $\CB_i$ are generated
	by exceptional subcollections of $(E_1,E_2,\ldots,E_n)$. Thus, it is enough to check semi-orthogonality between the objects of these collections. It remains to note that
\[
	\Ext^\bullet(E_p(k), E_q(l))=\Ext^\bullet(E_p(k-l),E_q).
\]
\end{proof}


\subsection{Borel--Bott--Weil Theorem and Littlewood--Richardson Rule} 
\label{sub:borel_bott_weil_theorem}

The Borel--Bott--Weil theorem is an extremely powerful tool that computes the cohomology of line bundles on the flag variety of a semisimple algebraic group.
It can also be used to compute the cohomology of equivariant vector bundles on Grassmannians. We restrict ourselves to the case of the group $\GL(V)$.

Let $V$ be a vector space of dimension $n$. We identify the weight lattice of the group $\GL(V)$ with $\ZZ^n$, taking the $k$-th fundamental weight $\pi_k$, which
is the highest weight of the representation $\Lambda^kV$, to the vector $(1,\ldots,1,0,\ldots,0)$ (where the first $k$ entries are equal to $1$ and the
other are zero). In this presentation the cone of dominant weights of $\GL(V)$ corresponds to the set of nonincreasing integer sequences
$\alpha=(a_1,a_2,\ldots,a_n)$, $a_1\geq a_2\geq\ldots\geq a_n$. For such $\alpha$ let $\Sigma^\alpha V$ denote the corresponding representation of $\GL(V)$.

\begin{remark}
	In the following we will use Young diagrams, which also represent nonincreasing positive finite integer sequences. From the very beginning we
	should warn the reader that sometimes our Young diagrams will have negative entries. On can think of negative rows as of boxes drawn to the left of some  
	chosen vertical zero axis.
\end{remark}

Similarly, given a rank $n$ vector bundle $E$ on a scheme $S$, one can take the corresponding principal $\GL(n)$-bundle and construct the vector bundle
$\Sigma^\alpha E$ associated with the $\GL(n)$ representation of highest weight $\alpha$.

The Weyl group $\BS_n$ of $\GL(n)$ acts naturally on the weight lattice $\ZZ^n$. Let $\ell:\BS_n\to\ZZ$ denote the standard length function. For every
$\alpha\in\ZZ^k$ there exists an element $\sigma\in\BS_n$ such that $\sigma(\alpha)$ is nonincreasing, which is unique if and only if all the
entries of $\alpha$ are distinct.

Let $X$ be the flag variety of $\GL(V)$, and let $L_\alpha$ denote the line bundle on $X$ corresponding to the weight $\alpha$ (thus, $L_{\pi_k}$ is
just the pullback of $\CO_{\PP(\Lambda^kV)}(1)$ under the natural projection $X\to\PP(\Lambda^kV)$). Denote by
\[
	\rho = (n,n-1,\ldots,1)
\]
the half sum of the positive roots of $\GL(V)$. The corresponding line bundle $L_\rho$ is the square root of the anticanonical line bundle.

The Borel--Bott--Weil theorem computes the cohomology of line bundles $L_\alpha$ on $X$.

\begin{theorem}[\cite{BBW}]\label{thm:bbw}
	Assume that all entries of $\alpha+\rho$ are distinct. Let $\sigma$ be the unique permutation such that $\sigma(\alpha+\rho)$ is strictly decreasing. Then
\[
	H^k(X,L_\alpha)=\begin{cases}
		\Sigma^{\sigma(\alpha+\rho)-\rho}V^* & \text{if } k=l(\sigma), \\
		0 & \text{otherwise}.
	\end{cases}
\]
	If at least two entries of $\alpha+\rho$ coincide then $H^\bullet(X, L_\alpha)=0$.
\end{theorem}

Now consider a Grassmannian $\Gr(k,V)$. Let $\CU\subset V\otimes\CO_{\Gr(k,V)}$ denote the tautological subbundle of rank~$k$. We have the following mutually
dual short exact sequences:
\[
0 \to \CU \to V\otimes\CO_{\Gr(k,V)} \to V/\CU \to 0,\qquad
0 \to \CU^\perp \to V^*\otimes\CO_{\Gr(k,V)} \to \CU^* \to 0,
\]
where $\CU^\perp = (V/\CU)^*$.
Note that $\Sigma^{1,1,\ldots,1}\CU^*\simeq\Sigma^{-1,-1,\ldots,-1}\CU^\perp$ is the positive generator of $\Pic\Gr(k,V)$. Let $\pi$ denote the canonical
projection $\pi:X\to\Gr(k,V)$ from the flag variety to the Grassmannian.

\begin{proposition}[\cite{KapranovGr}]
	Let $\beta\in\ZZ^k$ and $\gamma\in\ZZ^{n-k}$ be two nonincreasing integer sequences. Let $\alpha=(\beta,\gamma)\in\ZZ^n$ be their concatenation. Then we
	have $R\pi_*L_\alpha\simeq\Sigma^\beta\CU^*\otimes\Sigma^\gamma\CU^\perp$.
\end{proposition}

\begin{corollary}\label{crl:coh}
	If $\beta\in\ZZ^k$ and $\gamma\in\ZZ^{n-k}$ are two nonincreasing sequences and $\alpha=(\beta,\gamma)\in\ZZ^n$ then
\[
	H^\bullet(\Gr(k,V),\Sigma^\beta\CU^*\otimes\Sigma^\gamma\CU^\perp)\simeq H^\bullet(X,L_\alpha).
\]
\end{corollary}

Since every irreducible $\GL(V)$-equivariant vector bundle on $\Gr(k,V)$ is isomorphic to $\Sigma^\beta\CU^*\otimes\Sigma^\gamma\CU^\perp$ for some
nonincreasing $\beta\in\ZZ^k$ and $\gamma\in\ZZ^{n-k}$, a combination of Corollary~\ref{crl:coh} and Theorem~\ref{thm:bbw} allows to compute the
cohomology of any equivariant vector bundle on $\Gr(k,V)$.

In order to compute $\Ext$ groups between equivariant bundles on the Grassmannian one needs to use the Littlewood--Richardson rule~\cite{LR}. In the following we
will need a simple observation that follows directly from the Littlewood--Richardson rule.

\begin{lemma}\label{lm:lrrule}
	Let $\lambda$ and $\mu$ be two Young diagrams with $k$ rows. Then for any irreducible summand
\[
	\Sigma^\alpha\CU^*\subset\Sigma^\lambda\CU^*\otimes\Sigma^\mu\CU^*
\]
	one has
\[
	\lambda_i+\mu_k\leq\alpha_i\leq\lambda_1+\mu_i
\]
	for all $1\leq i\leq k$.
\end{lemma}


\subsection{Kapranov's exceptional collections} 
\label{sub:kapranov_s_exceptional_collections}

Let $X=\Gr(k,V)$ denote the Grassmannian of subspaces of dimension $k$ in an $n$-dimensional vector space $V$. Let $\CU$ denote the tautological
subbundle of rank $k$ on $X$ and let $\You_{n,k}$ denote the set of Young diagrams inscribed in a rectangle of size $k\times(n-k)$.

\begin{theorem}[\cite{KapranovGr}]
	The collection $\left\{\Sigma^\lambda\CU\mid \lambda\in\You_{n,k}\right\}$ (with any order refining the partial inclusion order~$\preceq$ on~$\You_{n,k}$)
	is a full exceptional collection in $\D^b(X)$. Moreover, this collection is strong, which implies that $\D^b(X)$ is equivalent to the homotopy 
	category of bounded complexes of sheaves consisting of finite direct sums of sheaves $\Sigma^\lambda\CU$ where $\lambda\in\You_{n,k}$.
\end{theorem}

Furthermore, there is a very useful spectral sequence described in the following theorem.

\begin{theorem}[\cite{KapranovGr}]\label{specseq}
	For every $\CF^\bullet\in\D^b(X)$ there is a generalized Beilinson spectral sequence
\[
E^{pq}_1=\bigoplus_{|\alpha|=-p}\HH^q(\CF^\bullet \otimes \Sigma^{\alpha^*}\CU^\perp)\otimes \Sigma^\alpha\CU \Rightarrow H^{p+q}(\CF^\bullet),
\]
where $\alpha$ runs over $\You_{n,k}$ and $\alpha^*$ denotes the transpose partition.
\end{theorem}

We should mention that Kapranov's original collection is based on the tautological bundle $\CU$ while we prefer to work with its dual $\CU^*$.
This does not change much as the duality functor is an anti-auto-equivalence of categories switching between $\CU$ and $\CU^*$. We hope that this will not lead
to any confusion.



\section{Combinatorics of Young diagrams} 
\label{sec:combinatorics_of_young_diagrams}

\subsection{Group actions on diagrams} 
\label{sub:group_actions_on_diagrams}

Let $\You_{n,k}$ denote the set of Young diagrams inscribed in a rectangle of size $k\times(n-k)$. These can be identified with nonincreasing integer sequences $\lambda=(\lambda_1,\lambda_2,\ldots,\lambda_k)$
such that $n-k\geq\lambda_1\geq\lambda_2\geq\ldots\geq\lambda_k\geq 0$. We call this description \textbf{usual}. One can also think of such diagrams as of
integer paths going from the lower left to the upper right corner of the rectangle that go only rightward and upward. Such a path consists of $n-k$
horizontal and $k$ vertical unit segments. Thus, one can think of $\You_{n,k}$ as of the set of binary sequences of length $n$ containing ``1'' exactly $k$ times.

Using the latter description, it is easy to construct a natural action of the group $\ZZ/n\ZZ$ on $\You_{n,k}$. This is just the cyclic action on the binary
sequences, where the generator $g$ of $\ZZ/n\ZZ$ acts by
\[
	g:a_1a_2\ldots a_n\ \mapsto\ a_na_1a_2\ldots a_{n-1},
\]
where $a_i\in\{0,1\}$. In the following we will call this action \textbf{cyclic}. For any $\lambda\in\You_{n,k}$ we denote its image under the action of the
generator $g$ by $\lambda'$ and call it the \textbf{shift} of $\lambda$. One can also describe $\lambda'$ in the usual presentation:
\[
\lambda' = \begin{cases}
	(\lambda_1+1,\lambda_2+1,\ldots,\lambda_k+1), & \text{if } \lambda_1<n-k, \\
	(\lambda_2,\lambda_3,\ldots,\lambda_k,0), & \text{if } \lambda_1=n-k.
\end{cases}
\]
The action of $g^d$ on $\lambda$ will be denoted by $\lambda^{(d)}$.

There is also an action of the group $\ZZ$ on all the diagrams with $k$ rows (and possibly negative entries) defined by
\[
	\lambda(t)=(\lambda_1+t,\lambda_2+t,\ldots,\lambda_k+t).
\]
We will say that $\lambda(t)$ is a \textbf{twist} of $\lambda$ by $t$.

\begin{definition}
	A diagram $\lambda\in\You_{n,k}$ is called \textbf{upper triangular} if it lies above the diagonal of the rectangle going from the upper right to the lower
	left corner. In the usual description it means that
\begin{equation}\label{eq:upper}
	\lambda_i\leq\frac{(n-k)(k-i)}{k}
\end{equation}
	for all $i=1,\ldots,k$. In a similar way one defines \textbf{lower triangular} diagrams. We denote these sets by $\Youu_{n,k}$ and $\Youl_{n,k}$ respectively.
\end{definition}

There is a nice geometric way to describe an orbit of the cyclic action. Take some diagram $\lambda\in\You_{n,k}$ and extend it $n$-periodically in both
directions. This is the same as extending the binary sequence representing~$\lambda$. To get any other diagram from the same orbit one should pick an
integer point on the extended path and draw a rectangle of size $k\times (n-k)$ using this point as the upper right corner. To see the cyclic action one should
move the point along the path in the south west direction. We will call the integer points on the extended path \textbf{vertices}.

The latter description of the orbits allows to prove easily the following lemma.
\begin{lemma}
	Every orbit of the cyclic action on $\You_{n,k}$ contains an upper triangular element.
\end{lemma}
\begin{proof}
	Draw all the lines with slope $k/(n-k)$ passing through all the vertices of the extended path. As the path is $n$-periodic, one will get at most $n$
	distinct lines and the path will lie above the lowest of them. Now draw the rectangle putting the upper right corner in any of the vertices lying on the
	lowest line. This will give the desired upper triangular element.
\end{proof}

A vertex on the extended path is called \textbf{u-admissible} (\textbf{l-admissible}) if it is the upper right corner of an upper (lower) triangular diagram.

There is an involution on the set of diagrams $\You_{n,k}$. Given a diagram $\lambda\in\You_{n,k}$, its \textbf{complement} $\lambda^c$ is defined in the
usual presentation by
\[
	\lambda^c=(n-k-\lambda_k, n-k-\lambda_{k-1},\ldots,n-k-\lambda_1).
\]
In the binary representation the sequence $a_1a_2\ldots a_{n-1}a_n$ is just mapped to its reverse $a_n a_{n-1}\ldots a_2a_1$.

One immediately sees that $(\lambda^c)^c=\lambda$. A nice fact is that taking the complement maps upper triangular diagrams to lower triangular and vice versa.
We will not need it in the following, but it is worth mentioning that $(\lambda^\prime)^c=(\lambda^c)^{(n-1)}$.

Given a diagram $\lambda$ with $k$ rows and arbitrary (possibly negative) entries one defines its negative by
\[
	-\lambda = (-\lambda_k,-\lambda_{k-1},\ldots,-\lambda_1).
\]
In these terms the complement can be defined as $\lambda^c=(-\lambda)(n-k)$ and $\Sigma^{\lambda}\CU=\Sigma^{-\lambda}\CU^*$.


\subsection{Orders and further characteristics} 
\label{sub:orders}

In the following we will need two order relations on $\You_{n,k}$. The first one is the partial inclusion order $\preceq$. Let us say that
\[
\lambda\preceq\mu\quad\text{if}\quad\lambda_i\leq\mu_i\text{ for all } i=1,\ldots,k.
\]
As usual, we say that $\lambda\prec\mu$ if $\lambda\preceq\mu$ and $\lambda\neq\mu$. The second one is
the lexicographical order $\leq$ which refines $\preceq$. Let us say that
\[
	\lambda<\mu\quad\text{if}\quad\lambda_i=\mu_i\text{ for } i=1,\ldots,t-1\quad\text{and}\quad \lambda_t<\mu_t\text{ for some }1\leq t\leq k.
\]

\begin{definition}
	An upper triangular diagram $\lambda\in\You_{n,k}$ is called \textbf{minimal} if it is the smallest among all the upper triangular elements in its orbit with
	respect to the lexicographical order. We denote the set of minimal upper triangular diagrams by $\Youmu_{n,k}$. Note that $\Youmu_{n,k}$ is indexing
	the orbits of the cyclic action. Given a diagram $\lambda\in\You_{n,k}$ we denote the length of its orbit by $o(\lambda)$.
\end{definition}

\begin{example}
	Consider the set $\You_{6,3}$. The orbits of the $\ZZ/6\ZZ$ action are the following:{\tiny
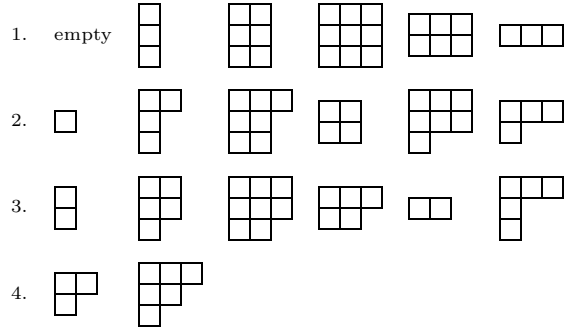
\[{\Yvcentermath1
	\begin{array}{lllllll}
		1. & \text{empty} & \yng(1,1,1) & \yng(2,2,2) & \yng(3,3,3) & \yng(3,3)   & \yng(3) \\ \\
		2. & \yng(1)      & \yng(2,1,1) & \yng(3,2,2) & \yng(2,2)   & \yng(3,3,1) & \yng(3,1) \\ \\
		3. & \yng(1,1)    & \yng(2,2,1) & \yng(3,3,2) & \yng(3,2)   & \yng(2)     & \yng(3,1,1) \\ \\
		4. & \yng(2,1)    & \yng(3,2,1) & & & & \\
	\end{array}
}\]}
	The set of upper triangular diagrams is:
\[{\Yvcentermath1
	\Youu_{n,k}=\mbox{\tiny $\left\{\text{empty},\ \yng(1),\ \yng(1,1),\ \yng(2),\  \yng(2,1)\right\}$}.
}\]
	The set of minimal upper triangular diagrams is:
\[{\Yvcentermath1
	\Youmu_{n,k}=\mbox{\tiny $\left\{\text{empty},\ \yng(1),\ \yng(1,1),\  \yng(2,1)\right\}$}.
}\]
	Upper triangular diagrams {\tiny $\yng(1,1)$} and {\tiny $\yng(2)$} lie in the same orbit and the first one is the minimal one.
\end{example}

There is a natural associative noncommutative operation
\[
\oplus:\You_{n,k}\times \You_{m,l}\to \You_{n+m,k+l}
\]
which is the concatenation of the binary sequences representing the diagrams. Namely, for diagrams $a\in\You_{n,k}$ and $b\in\You_{m,l}$ one has
\[
	a\oplus b = a_1 a_2\ldots a_n b_1 b_2\ldots b_m
\]
in the binary presentation $a=a_1 a_2\ldots a_n$, $b=b_1 b_2\ldots b_m$ and
\[
	a\oplus b = (\mu_1+(n-k),\mu_2+(n-k),\ldots,\mu_l+(n-k),\lambda_1,\lambda_2, \ldots,\lambda_k)
\]
in the usual presentation $a=(\lambda_1,\lambda_2,\ldots,\lambda_k)$, $b=(\mu_1,\mu_2,\ldots,\mu_l)$.

\begin{lemma}
	Given $a,a^\prime\in\You_{n,k}$ and $b,b^\prime\in\You_{m,l}$, we have the following:
	\begin{itemize}
		\item $b\oplus a<b^\prime\oplus a^\prime$ if and only if $a<a^\prime$, or $a=a^\prime$ and $b<b^\prime$;
		\item $b\oplus a\preceq b^\prime\oplus a^\prime$ if and only if $a\preceq a^\prime$ and $b\preceq b^\prime$.
	\end{itemize}
\end{lemma}
\begin{proof}
	Obvious.
\end{proof}

\newpage
Finally, we will need several characteristics of different types of diagrams.
\begin{itemize}
	\item Given a diagram $\lambda\in\You_{n,k}$, its \textbf{slope} is defined by $s(\lambda)=k/(n-k)$. One immediately checks that the sum of two upper (resp.
	lower) triangular diagrams of equal slope is again upper (resp.~lower) triangular.
	\item Given a diagram $\lambda\in\Youu_{n,k}$, let $r(\lambda)$ denote the length of the path going from the upper right corner of the rectangle to the
	rightmost vertex of the path that lies on the diagonal and is distinct from the initial one. For example, if $\lambda$ is strictly upper triangular,
	$r(\lambda)=n$.
	\item Given a diagram $\lambda\in\Youl_{n,k}$, let $l(\lambda)$ denote the length of the path going from the lower left corner of the rectangle to the
	leftmost vertex of the path that lies on the diagonal and is distinct from the initial one. In particular, note that $\lambda^c\in\Youu_{n,k}$ and
	$l(\lambda)=r(\lambda^c)$.
	\item Given a diagram $\lambda\in\You_{n,k}$, let $d(\lambda)$ denote the smallest $d\geq 0$, such that $\lambda^{(d)}$ is lower triangular. This is the
	same as the length of the path going from the upper right corner of the rectangle to the rightmost l-admissible vertex of the diagram. In particular, if
	$\lambda\in\Youl_{n,k}$ then $d(\lambda)=0$.
	\item Given a diagram $\lambda\in\You_{n,k}$, let $e(\lambda)$ denote the length of the path going from the lower left corner to the leftmost
	l-admissible vertex that is distinct from the initial one. In particular, $e(\lambda)>0$ and whenever $\lambda\in\Youl_{n,k}$, we have
	$e(\lambda)=l(\lambda)$.
\end{itemize}



\section{Lefschetz decompositions for Grassmannians} 
\label{sec:lefschetz_decompositions_for_grassmannians}

From this moment $X=\Gr(k,V)$ will denote the Grassmannian of $k$-dimensional subspaces of an $n$-dimensional vector space $V$. Let $\CU\subset V\otimes\CO_X$
denote the tautological subbundle of rank $k$. We have the following (mutually dual) short exact sequences of vector bundles on $X$:
\[
0 \to \CU \to V\otimes\CO_X \to V/\CU \to 0,\qquad
0 \to \CU^\perp \to V^*\otimes\CO_X \to \CU^* \to 0,
\]
where $\CU^\perp=(V/\CU)^*$. Recall that $\Sigma^{1,\ldots,1}\CU^* \simeq \Sigma^{-1,\ldots,-1}\CU^\perp \simeq \CO_X(1)$ is the positive generator of
$\Pic X$.

\subsection{Statement of the main result} 
\label{sub:statement_of_the_main_result}

We introduce two collections of subcategories in $\D^b(X)$. The first one is defined by
\[
	\CA_i=\left<\Sigma^\lambda\CU^* \mid \lambda\in\Youmu_{n,k},\ i<o(\lambda)\right>,\text{ for } i=0,\ldots,n-1.
\]
The second one is the following:
\[
	\CB_i=\left<\Sigma^\lambda\CU^* \mid \lambda\in\Youu_{n,k},\ i<r(\lambda)\right>,\text{ for } i=0,\ldots,n-1.
\]
One immediately notes that $\CA_0\supset\CA_1\supset\ldots\supset\CA_{n-1}$ and $\CB_0\supset\CB_1\supset\ldots\supset\CB_{n-1}$.

We expect that both of these collections give a Lefschetz decomposition of $\D^b(X)$. However, at the moment we can prove fullness only for $\CB_i$.

The following two theorems are the main results of this paper.

\begin{theorem}\label{theorem:brav}
	The categories $\CB_i$ form a Lefschetz decomposition of $\D^b(X)$. The exceptional collection $\left(\Sigma^\lambda\CU^*\mid \lambda\in\Youu_{n,k}\right)$
	is a Lefschetz basis of this decomposition with the support function $r(\lambda)$.
\end{theorem}

\begin{example}\label{ex:gr36brav}
	Consider the case $X=\Gr(3,6)$. We get a Lefschetz basis
\[
	(\CO_X, \CU^*, \Lambda^2\CU^*, S^2\CU^*, \Sigma^{(2,1,0)}\CU^*)
\]
	with the following values of the support function:
\[
	(6, 6, 4, 2, 2).
\]
	In other words, there is a full Lefschetz exceptional collection
\[
	\left(
	\begin{array}{rrrrrr}
		\Sigma^{(2,1,0)}\CU^* & \Sigma^{(2,1,0)}\CU^*(1) & & & & \\
		S^2\CU^* & S^2\CU^*(1) & & & & \\
		\Lambda^2\CU^* & \Lambda^2\CU^*(1) & \Lambda^2\CU^*(2) & \Lambda^2\CU^*(3) & & \\
		\CU^* & \CU^*(1) & \CU^*(2) & \CU^*(3) & \CU^*(4) & \CU^*(5) \\
		\CO_X & \CO_X(1) & \CO_X(2) & \CO_X(3) & \CO_X(4) & \CO_X(5) \\
	\end{array}
	\right).
\]
	Here objects standing in the same column generate blocks of the corresponding decomposition.
\end{example}

\begin{theorem}\label{theorem:me}
	The categories $\CA_i(i)$ are semi-orthogonal. In other words, there is a Lefschetz decomposition
\[
	\left<\CA_0,\CA_1(1),\ldots,\CA_{n-1}(n-1)\right>=\CA\subset\D^b(X)
\]
	of some full triangulated subcategory $\CA\subset\D^b(X)$ with a Lefschetz basis given by the exceptional collection $\left(\Sigma^\lambda\CU^*\mid \lambda\in\Youmu_{n,k}\right)$
	and the support function $o(\lambda)$.
\end{theorem}

Theorem~\ref{theorem:me} is followed by two important conjectures.

\begin{conjecture}\label{conj:full}
	The categories $\CA_i(i)$ generate $\D^b(X)$. In other words, the categories $\CA_i$ form a Lefschetz decomposition of $\D^b(X)$.
\end{conjecture}

\begin{example}
	In the case $X=\Gr(3,6)$ this would give a Lefschetz exceptional collection
\[\left(
	\begin{array}{rrrrrr}
		\Sigma^{(2,1,0)}\CU^* & \Sigma^{(2,1,0)}\CU^*(1) & & & & \\
		\Lambda^2\CU^* & \Lambda^2\CU^*(1) & \Lambda^2\CU^*(2) & \Lambda^2\CU^*(3) & \Lambda^2\CU^*(4) & \Lambda^2\CU^*(5) \\
		\CU^* & \CU^*(1) & \CU^*(2) & \CU^*(3) & \CU^*(4) & \CU^*(5) \\
		\CO_X & \CO_X(1) & \CO_X(2) & \CO_X(3) & \CO_X(4) & \CO_X(5) \\
	\end{array}\right).
\]
	On can compare this collection to the one from Example~\ref{ex:gr36brav} and see that the objects $S^2\CU^*$ and $S^2\CU^*(1)$ were replaced by
	$\Lambda^2\CU^*(4)$ and $\Lambda^2\CU^*(5)$, making first block of the collection smaller.
\end{example}

As we mentioned before, an important question is to construct not arbitrary, but minimal Lefschetz decompositions. We state the following.

\begin{conjecture}\label{conj:min}
	The categories $\CA_i$ form a minimal Lefschetz decomposition of $\D^b(X)$.
\end{conjecture}

\begin{remark}\label{rem:comp}
	We consider two Lefschetz decompositions $\CA_i$ and $\CB_i$ of the category
	$\D^b(X)$. Remark that $\CA_0\subset\CB_0$, as the exceptional collection generating $\CA_0$ is a subcollection of the collection generating $\CB_0$.
	Indeed, the first one is indexed by minimal upper triangular diagrams, and the second one by upper triangular diagrams. Thus, the decomposition $\CA_i$ is
	smaller or equal to the decomposition $\CB_i$ with respect to the inclusion order on the first blocks.
	
	Moreover, there are only two cases when $\CA_0=\CB_0$. This happens if and only if $k$ and $n$ are coprime, or $k=2$. In these cases every orbit of the
	shift action on $\You_{n,k}$ has a single upper triangular element, which is automatically minimal, and for any $\lambda\in\Youu_{n,k}$ we have
	$r(\lambda)=o(\lambda)$. Thus, $\CA_i=\CB_i$ for $i=0,\ldots,n-1$, which confirms the fact that every Lefschetz decomposition is determined by its first
	block.
\end{remark}

Conjecture~\ref{conj:min} seems to be hard, as we do not know any natural way to prove minimality of Lefschetz collections. However, there are some cases when it
immediately holds.

\begin{proposition}
	We have the following:
	\begin{enumerate}
		\item Conjectures~\ref{conj:full} and \ref{conj:min} hold when $n$ and $k$ are coprime.
		\item If $k=p$ is a prime number, then Conjecture~\ref{conj:full} implies Conjecture~\ref{conj:min}.
	\end{enumerate}
\end{proposition}
\begin{proof}
	It is easy to see that the number of blocks in a Lefschetz decomposition is bounded above by the index of our variety. Indeed, if $\omega_X\simeq\CO_X(-n)$
	then by Serre duality we have
\[
	\Hom(E(n), E(n))\simeq \Ext^d(E(n), E)^*
\]
	for any object $E\in\D^b(X)$, where $d$ is the dimension of the variety $X$. The first group has a distinguished nonzero element
\[
	id_{E(n)}\in\Hom(E(n),E(n))
\]
	which shows that $\Ext^d(E(n), E)$ is nonzero. Thus, the object $E(n)$ can not be left orthogonal to $E$ for any $E\in\D^b(X)$.
	
	In the case of the Grassmannian $\Gr(k,V)$ the index is equal to $n=\dim V$. Thus, all our decompositions already have the maximal possible number of blocks.

	We have seen in Remark~\ref{rem:comp} that the categories $\CA_i$ coincide with $\CB_i$ whenever $k$ and $n$ are coprime, thus, form a Lefschetz
	decomposition. Moreover, in this case the decomposition is rectangular with the maximal possible number of blocks, thus, minimal. This proves the first
	statement.
	
	To prove the second statement one should first note that in the case $k=p$ and $k$ divides $n$ there is a single short orbit represented by the minimal
	upper triangular diagram
\[
	\left(\frac{(n-k)}{k}(k-1),\ldots, \frac{(n-k)}{k}, 0\right).
\]
	Then, the number $N$ of objects in any full exceptional collection is equal to the rank of $K_0(X)=\binom{n}{k}$. If $k=p$ is a prime that divides $n$, the
	number $N$ is not divisible by $n$. That means that there is no rectangular Lefschetz exceptional collection/decomposition in this case. At the same time, the
	number of objects in the Lefschetz basis associated to $\CA_i$ equals the expected minimum $\left\lceil\frac{N}{k}\right\rceil$.
\end{proof}

Finally, we should mention that the second Lefschetz decomposition has been discovered by Chris Brav and Hugh Thomas~(\cite{Brav}). However, they only proved
exceptionality for the corresponding Lefschetz basis and have never published their result. In the present paper we prove semi-orthogonality of both
decompositions and fullness of the second one.


\subsection{Semi-orthogonality} 
\label{sub:semi_orthogonality}

In this section we prove semi-orthogonality for the categories $\CA_i(i)$ and $\CB_j(j)$:
\[
	\CA_j(j)\subset\vphantom{\CA_i(i)}^\perp\CA_i(i)\quad\text{and}\quad \CB_j(j) \subset\vphantom{\CB_i(i)}^\perp\CB_i(i)\qquad\text{for }0\leq i<j\leq n-1. 
\]

\begin{proposition}\label{prop:exc_b}
	The categories $\CB_i(i)$ are semi-orthogonal:
\[
	\CB_j(j) \subset\vphantom{\CB_i(i)}^\perp\CB_i(i)\quad\text{for}\quad 0\leq i<j\leq n-1.
\]
\end{proposition}
\begin{proof}
	Apply the second part of the criterion given in Lemma~\ref{lemma:basis}. It is sufficient to check that for all
	$\lambda,\mu\in\Youu_{n,k}$ and $0<t<r(\lambda)$ one has
\[
	\Ext^\bullet(\Sigma^\lambda\CU^*(t),\Sigma^\mu\CU^*) = 0.
\]

	First of all, we have
\[
	\Ext^\bullet(\Sigma^\lambda\CU^*(t),\Sigma^\mu\CU^*) = H^\bullet(X, \Sigma^\lambda\CU \otimes \Sigma^\mu\CU^*(-t)).
\]
	To compute the latter we apply the Littlewood--Richardson rule. Let $\Sigma^\alpha\CU^*$ be a vector bundle appearing in the decomposition of
	$\Sigma^\lambda\CU\otimes\Sigma^\mu\CU^*(-t)=\Sigma^{-\lambda}\CU^*\otimes\Sigma^\mu\CU^*(-t)$. By Lemma~\ref{lm:lrrule}
\begin{equation}\label{eq:alpha_summand}
	-\lambda_{k+1-i}-t\leq \alpha_i\leq \mu_i-t.
\end{equation}
	Recall that both $\lambda$ and $\mu$ are upper triangular, which by definition means that
\[
0\leq \lambda_i,\mu_i\leq \frac{(n-k)(k-i)}{k},
\]
	see~\eqref{eq:upper}. Combining the last two inequalities, we get
\begin{equation}\label{ineq}
	-\frac{(i-1)(n-k)}{k}-t\leq \alpha_i\leq \frac{(k-i)(n-k)}{k}-t.
\end{equation}

	In order to compute $H^\bullet(X,\Sigma^\alpha\CU^*)$ one should apply the Borel--Bott--Weil theorem. We want the cohomology to be zero. It means that some
	of the terms in the sequence
\begin{equation}\label{sumsq}
	\alpha+\rho=(n+\alpha_1,n-1+\alpha_2,\ldots,n-k+1+\alpha_k,n-k,\ldots,2,1)
\end{equation}
	coincide. As the sequence $(\alpha_1,\alpha_2,\ldots,\alpha_k)$ is nonincreasing, the first $k$ terms of sequence~\eqref{sumsq} are distinct and
	decreasing. The last $(n-k)$ terms of sequence~\eqref{sumsq} are also distinct and decreasing. Thus, some of the terms in~\eqref{sumsq}
	coincide if and only if $n+1-i+\alpha_i$ belongs to the segment $\left[1,n-k\right]$ for some $1\leq i\leq k$.
	
	Imagine that all the terms in~\eqref{sumsq} are distinct, which means that
\begin{equation}\label{eq:alpha}
	n+1-j+\alpha_j\geq n-k+1,\qquad n-j+\alpha_{j+1}\leq 0
\end{equation}
	for some $1\leq j\leq k-1$. Here we used the facts that
\[
	\alpha_1\geq-t>-n\ \Rightarrow\ n+\alpha_1>0\quad\text{and}\quad\alpha_k \leq -t < 0\ \Rightarrow\ n-k+1+\alpha_k\leq n-k.
\]
	Now, combine inequalities~\eqref{eq:alpha} with~\eqref{ineq}:
\begin{eqnarray*}
	\alpha_j\geq j-k & \Rightarrow & \frac{(k-j)(n-k)}{k}-t\geq j-k,\\
	\alpha_{j+1}\leq j-n & \Rightarrow & -\frac{j(n-k)}{k}-t\leq j-n.
\end{eqnarray*}
	One gets the following:
\[
	t\leq\frac{n(k-j)}{k},\ t\geq\frac{n(k-j)}{k}\quad\Rightarrow\quad t=\frac{n(k-j)}{k}.
\]

	Note that $t$ should be an integer number. Combine inequalities $n-j+\alpha_{j+1}\leq 0$ and~\eqref{eq:alpha_summand} for $\lambda_{k-j}$:
\[
	-\lambda_{k-j}-\frac{n(k-j)}{k}\leq j-n\quad\Rightarrow\quad \lambda_{k-j}\geq\frac{(n-k)j}{k}.
\]

	Recall that $\lambda$ is upper triangular, thus, $\lambda_{k-j}\leq (n-k)j/k$. We deduce that
\[
	\lambda_{k-j} = \frac{(n-k)j}{k}
\]
	i.e. that $\lambda$ meets the diagonal in the point in the $(k-j)$-th row. Remark that the distance from the upper right corner of the rectangle to this
	point is exactly
\[
	(k-j)+\left((n-k)-\frac{(n-k)j}{k}\right) = \frac{n(k-j)}{k} = t.
\]
	This means that $t\geq r(\lambda)$, which contradicts the assumptions from the statement.
\end{proof}

To prove semi-orthogonality of the categories $\CA_i(i)$ we will need the following simple observation.

\begin{lemma}\label{nonzlm}
	Let $\Sigma^\alpha\CU^*$ be an irreducible summand in the decomposition of $\Sigma^\lambda\CU\otimes\Sigma^\mu\CU^*$ for some $\lambda,\mu\in\You_{n,k}$.
	If all the terms $\alpha_i\geq 0$ are nonnegative, then $\lambda\preceq\mu$.
\end{lemma}
\begin{proof}
	Remark that $\Sigma^\alpha\CU^*\subset \Sigma^\lambda\CU\otimes \Sigma^\mu\CU^*$ if and only if
	$\Sigma^\mu\CU^*\subset \Sigma^\alpha\CU^*\otimes\Sigma^\lambda\CU^*$. Both $\lambda$ and $\alpha$ have only non-negative entries. Thus, by Lemma~\ref{lm:lrrule}
	for every summand $\Sigma^\beta\CU^*$ in $\Sigma^\alpha\CU^*\otimes\Sigma^\lambda\CU^*$ one has $\lambda\preceq\beta$.
\end{proof}

We are ready to give a proof of Theorem~\ref{theorem:me}. Recall that it states that the categories $\CA_i(i)$ are semi-orthogonal:
\[
	\CA_j(j) \subset\vphantom{\CA_i(i)}^\perp\CA_i(i)\quad\text{for}\quad 0\leq i<j\leq n-1.
\]

\begin{proof}[Proof of Theorem~\ref{theorem:me}]
	Let us repeat the proof of Proposition~\ref{prop:exc_b} keeping in mind that $\lambda,\mu\in\Youmu_{n,k}$. It works until the end where we use inequality
	$t<r(\lambda)$, which is replaced by a weaker assumption $t<o(\lambda)$.
	
	Thus, we can assume that 
\[
	 t=\frac{n(k-j)}{k}\in\ZZ\quad \text{and} \quad\lambda_{k-j}=\frac{j(n-k)}{k}.
\]
	Analogously one can show that $\mu_j=(k-j)(n-k)/k$.
	
	The latter equalities mean that the upper triangular diagrams $\lambda,\mu$ meet the diagonal and can be written in the form $\mu=b\oplus a$ and
	$\lambda=b^\prime\oplus a^\prime$ for some $a,b^\prime\in\You_{jn/k,j}$ and $a^\prime,b\in\You_{(k-j)n/k,k-j}$. In follows from the Littlewood--Richardson rule
	that the $i$-th row of $\mu$ contributes to the rows of the summands in the tensor product with numbers greater or equal to $i$. Now, apply Lemma~\ref{nonzlm}
	and deduce that $a\succeq b^\prime$ and, similarly, $a^\prime\succeq b$.
	
	Note that $a\oplus b$ (resp. $a^\prime\oplus b^\prime$) lies in the same orbit as $\mu$ (resp. $\lambda$).
	
	We will need the following observation. The fact that $\lambda$ is minimal upper triangular implies that $\lambda=b^\prime\oplus a^\prime\leq a^\prime\oplus b^\prime$.
	However, inequality $t<o(\lambda)$ implies that $b^\prime\oplus a^\prime$ and $a^\prime\oplus b^\prime$ can not be equal and this inequality is strict:
\begin{equation}\label{eq:permless}
	b^\prime\oplus a^\prime < a^\prime\oplus b^\prime.
\end{equation}
	
	Consider the case $a\succ b^\prime$ and $a^\prime\succ b$. Then one has
\[
	\lambda=b^\prime\oplus a^\prime < a^\prime\oplus b^\prime < b\oplus a \leq a\oplus b < b^\prime\oplus a^\prime = \lambda,
\]
	where the first and the third inequalities come from minimality of $\lambda$ and $\mu$, the second follows from $a\succ b^\prime$ and the fourth from
	$a^\prime\succ b$. We get $\lambda<\lambda$, which is a contradiction.
	
	In the following given a diagram $\alpha\in\You_{p,q}$ by the height function we mean $h(\alpha)=q$. We will also use notation $m^p=\underbrace{m\oplus m\oplus\ldots\oplus m}_{p\ \text{times}}$.
	
	\emph{Case 1:} $b^\prime\prec a$, $a^\prime=b$ and $h(b)\geq h(a)$. We still have inequalities
\begin{equation}\label{eq:first_case}
	\lambda=b^\prime\oplus b < b\oplus b^\prime < b\oplus a=\mu\leq a\oplus b.
\end{equation}
	Particularly, $b^\prime\oplus b<b\oplus a\leq a\oplus b$, which implies that $b=c\oplus a$ and $\lambda = b^\prime\oplus c\oplus a$.
	As $b^\prime\prec a$, we get $b\oplus b^\prime < b^\prime\oplus c\oplus a=\lambda$, which contradicts inequality~\eqref{eq:first_case}.
	
	\emph{Case 2:} $b^\prime\prec a$, $a^\prime=b$ and $h(b)<h(a)$. As in the first case, consider inequalities
\[
	b^\prime\oplus b<b\oplus a\leq a\oplus b.
\]
	This implies that $a=c\oplus b$ and $\mu=b\oplus c\oplus b$. The diagrams $c\oplus b^2$ and $b^2\oplus c$ are upper triangular and lie in the orbit of $\mu$.
	The latter is minimal, thus $\mu\leq b^2\oplus c$ and $\mu\leq c\oplus b^2$. The last two inequalities imply that $c\oplus b=b\oplus c$, which means
	that there exists a diagram $m$ with the same slope as $b$, such that $b=m^p$ and $\mu=m^{p+q}$.
	
	Now consider inequalities
\[
	\lambda=b^\prime\oplus m^p<m^p\oplus b^\prime\leq m^q\oplus m^p=\mu.
\]
	As $h(b^\prime)=h(a)>h(b)=h(m^p)$, we see that $b^\prime=d\oplus m^p$. One can repeat the argument and continue extracting summands equal to $m^p$, until ends up with
\[
	\lambda=e\oplus m^{sp}\oplus m^p<m^p\oplus e\oplus m^{sp}\leq m^r\oplus m^p\oplus m^{sp}=\mu,
\]
	where $0<h(e)=h(m^r)\leq h(m^p)$ and $s$ is a positive integer. From the last inequality one gets that $e<m^r$. Meanwhile, $r\leq p$, thus
	$m^p\oplus m^{sp}\oplus e<e\oplus m^{sp}\oplus m^p=\lambda$. This once again contradicts minimality of $\lambda$.
	
	The case $b\prec a^\prime$, $a=b^\prime$ is treated similarly.
	
	Finally, note that the case $b=a^\prime$ and $a=b^\prime$ is not possible: we have shown that
\[
	\lambda=b^\prime\oplus a^\prime<a^\prime\oplus b^\prime = b\oplus a = \mu.
\]
	But in this case $\lambda$ and $\mu$ lie in the same orbit and should be equal, as they are supposed to be minimal upper triangular. This finishes the proof.
\end{proof}



\section{Fullness} 
\label{sec:fullness}

The goal of this section is to prove fullness of the decomposition
\[
	\left<\CB_0,\CB_1(1),\ldots,\CB_{n-1}(n-1)\right>.
\] We will need a special
kind of exact complexes that we call \textbf{staircase}.

\subsection{Staircase complexes} 
\label{sub:staircase_resolutions}

Let us start with the following lemma.

\begin{lemma}\label{extlemma}
	Let $\lambda\in\You_{n,k}$ be a diagram with $\lambda_1=n-k$. Then we have
\[
\Ext^p(\Sigma^\lambda\CU^*,\Sigma^{\lambda^\prime}\CU^*(-1)) = \begin{cases}
	\kk, & \text{if } p=n-k, \\
	0, & \text{otherwise}.
\end{cases}
\]
\end{lemma}
\begin{proof}
	The argument is similar to the proof of Proposition~\ref{prop:exc_b}.
	
	First of all, we have
\[
	\Ext^\bullet(\Sigma^\lambda\CU^*,\Sigma^{\lambda^\prime}\CU^*(-1)) = H^\bullet(X,\Sigma^\lambda\CU\otimes\Sigma^{\lambda^\prime}\CU^*(-1)) =
	H^\bullet(X,\Sigma^{-\lambda}\CU^*\otimes\Sigma^{\lambda^\prime}\CU^*(-1)).
\]
	
	To compute the latter we apply the Littlewood--Richardson rule. Let $\Sigma^\alpha\CU^*$ be an irreducible summand in the decomposition of
	$\Sigma^\lambda\CU\otimes\Sigma^{\lambda^\prime}\CU^*(-1)$. As we want $H^\bullet(X,\Sigma^\alpha\CU^*)$ to be nonzero, by the Borel--Bott--Weil theorem
	we require all the entries of 

\begin{equation}\label{sumeq}
	\alpha+\rho=(n+\alpha_1,n-1+\alpha_2,\ldots,n-k+1+\alpha_k,n-k,\ldots,2,1)	
\end{equation}
	to be distinct. As the sequence $\alpha_i$ is nonincreasing, the first $k$ and the last $n-k$ terms in~\eqref{sumeq} are distinct. Thus, it is sufficient to check that
\begin{equation}\label{nullint}
	\{n+\alpha_1,n-1+\alpha_2,\ldots,n-k+1+\alpha_k\}\cap\{n-k,\ldots,2,1\}=
	\emptyset.
\end{equation}

	We have obvious inequalities $\alpha_i\geq-(n-k)-1$ coming from the Littlewood--Richardson rule. Thus, all the elements in $(n+a_1,\ldots,n-k+1+a_k)$
	are nonnegative. The last sequence is nonincreasing, which implies that if~\eqref{nullint} holds, then
\[
	\text{either}\quad \alpha_i\geq-1\text{ for }i=1,\ldots,k-1\text{ and }
	\alpha_k=-(n-k)-1\quad\text{or}\quad \alpha_i\geq 0\text{ for all } i=1,\ldots,k.
\]
	It follows from Lemma~\ref{lm:lrrule} that $\alpha_k\leq (-\lambda)_k+\lambda^\prime_1 = -\lambda_1+\lambda_2-1\leq -1$. Thus, the second
	case is not possible and $\Sigma^{\bar{\alpha}}\CU^*$ for $\bar{\alpha}=(-1,\ldots,-1,-(n-k)-1)$ is contained in
	$\Sigma^\lambda\CU\otimes\Sigma^{\lambda^\prime}\CU^*(-1)$ with multiplicity 1, being the only term in the decomposition with nontrivial cohomology.
	
	Now we apply the Borel--Bott--Weil theorem to compute $H^\bullet(X,\Sigma^{-1,\ldots,-1,-(n-k)-1}\CU^*)$. We have
\[
\bar{\alpha}+\rho=(n-1,n-2,\ldots,n-k+1,0,n-k,\ldots,2,1).
\]
	Take $\sigma$ to be the cyclic permutation of the last $(n-k+1)$ entries. Then we have $\ell(\sigma)=n-k$ and $\sigma(\bar{\alpha}+\rho)=(n-1,\dots,1,0)$,
	$\sigma(\bar{\alpha}+\rho)-\rho=(-1,\dots,-1,-1)$.
	
	Finally, we get
\[
\Ext^p(\Sigma^\lambda\CU^*,\Sigma^{\lambda^\prime}\CU^*(-1)) =
H^p(X,\Sigma^{\bar{\alpha}}\CU^*) = \begin{cases}
	\Sigma^{-1,\ldots,-1}V^*\simeq\Lambda^n V, & \text{if } p=n-k, \\
	0, & \text{otherwise}.
\end{cases}
\]
	This finishes the proof.
\end{proof}

\begin{remark}
	If one takes $\lambda\in\You_{n,k}$ with $\lambda_1<n-k$, then $\Sigma^{\lambda^\prime}\CU^*(-1)\simeq\Sigma^{\lambda}\CU^*$. This means that
\[
	\Ext^p(\Sigma^\lambda\CU^*,\Sigma^{\lambda^\prime}\CU^*(-1)) = 
	\Ext^p(\Sigma^\lambda\CU^*,\Sigma^\lambda\CU^*) = \begin{cases}
		\kk, & \text{if } p=0, \\
		0, & \text{otherwise},
	\end{cases}
\]
	the bundle $\Sigma^\lambda\CU^*$ being exceptional. Combining with Lemma~\ref{extlemma} one can state that there is always a unique canonical
	extension between the bundles $\Sigma^\lambda\CU^*$ and $\Sigma^{\lambda^\prime}\CU^*(-1)$, either in degree $0$ or $n-k$.

	Moreover, the composition
\[
	\Sigma^\lambda\CU^*\to\Sigma^{\lambda^\prime}\CU^*(-1)[d_1]\to \Sigma^{\lambda^{\prime\prime}}\CU^*(-2)[d_2]\to\ldots\to \Sigma^{\lambda^{(n)}}\CU^*(-n)[d_n]=\Sigma^\lambda\CU^*(-n)[k(n-k)]
\]
	is nontrivial in $\Ext^{k(n-k)}(\Sigma^\lambda\CU^*,\Sigma^\lambda\CU^*(-n)) \simeq \Hom(\Sigma^\lambda\CU^*, \Sigma^\lambda\CU^*)^*=\kk$.
\end{remark}

It turns out that the canonical extension from Lemma~\ref{extlemma} can be realized explicitly.

\begin{proposition}\label{les}
	Let $\lambda\in\You_{n,k}$ be a diagram with $\lambda_1=n-k$. Then there is a long exact sequence
\begin{equation}\label{eq:staircase}
	0\to\Sigma^{\lambda^\prime}\CU^*(-1)\to \Lambda^{\nu_{n-k}}V^*\otimes\Sigma^{\mu_{n-k}}\CU^*\to \ldots\to \Lambda^{\nu_2}V^*\otimes\Sigma^{\mu_2}\CU^*\to \Lambda^{\nu_1} V^*\otimes\Sigma^{\mu_1}\CU^*\to \Sigma^\lambda\CU^*\to 0
\end{equation}
for some $0\leq\nu_i< n$ and $\mu_i\in\You_{n,k}$ that will be described later.
\end{proposition}
\begin{proof}
	It will be convenient to construct the dual long exact sequence. Let us apply Theorem~\ref{specseq} to the complex consisting of a single sheaf
	$\Sigma^{\lambda^\prime}\CU(1)$. One should compute
\[
E^{pq}_1=\bigoplus_{|\alpha|=-p}\HH^q(\Sigma^{\lambda^\prime}\CU(1) \otimes \Sigma^{\alpha^*}\CU^\perp)\otimes \Sigma^\alpha\CU.
\]

	First, let us find all $\alpha\in\You_{n,k}$ such that $\HH^\bullet(\Sigma^{\lambda^\prime}\CU(1) \otimes \Sigma^{\alpha^*}\CU^\perp)$
	is nonzero. The Borel--Bott--Weil theorem states that this is the case if and only if all the entries in the sequence
\[
	(n-\lambda^\prime_k+1, (n-1)-\lambda^\prime_{k-1}+1,\ldots, (n-k+1)-\lambda^\prime_1+1, n-k+\alpha^*_1, \ldots, 2+\alpha^*_{n-k-1}, 1+\alpha^*_{n-k})
\]
are distinct. Recall that $\lambda^\prime=(\lambda_2,\lambda_3,\ldots, \lambda_{n-k},0)$. Thus, one can rewrite the previous sequence as
\begin{equation}\label{bbwsq1}
	(n+1,n-\lambda_k,\ldots,n-k+2-\lambda_2,n-k+\alpha^*_1, \ldots, 2+\alpha^*_{n-k-1}, 1+\alpha^*_{n-k}).
\end{equation}

	Assume that all the entries in sequence~\eqref{bbwsq1} are distinct. Then after a suitable permutation they will form a strictly decreasing sequence. Since
	all the terms except the first are in $[1,n]$, the permuted sequence will be
\begin{equation}\label{eq:no_j}
	(n+1,n,\ldots,n-j+1,n-j-1,\ldots,1)
\end{equation}
	for some $0\leq j\leq n-1$. Moreover, $n-j\neq n-k+t-\lambda_t\Leftrightarrow j\neq k+\lambda_t-t$ for all $2\leq t\leq k$. Thus, we have $n-k+1$ choices for
	$j$ and each of the choices gives a unique $\alpha^*$ which is obtained by permuting sequence~\eqref{eq:no_j} in such a way that the first $k$ places
	are occupied by $n+1$ and $n-k+t-\lambda_t$ and the last $n-k$ entries decrease. Finally, remark that $\alpha^*=\lambda^*$ corresponds to $j=n-1$.
	
	This allows us to write down all the possible $\alpha$:
\begin{enumerate}
	\item[(0):] $\alpha^* = \mu_0 = \lambda^*$,
	\item[(i):] $\alpha^* = \mu_i = (\lambda^*_1,\ldots,\lambda^*_{n-k-i}, \lambda^*_{n-k-i+2}-1,\ldots,\lambda^*_{n-k}-1,0)$ for $i=1,\ldots,n-k$.
\end{enumerate}

	For each case we need the values $p$ and $q$. Recall that $p$ is equal to $-|\alpha|$ and $q$ is equal to the number of inversions in sequence~\eqref{bbwsq1}.
	These are easy to compute and one gets the following result:
\begin{enumerate}
	\item[(0):] $p_0=-|\lambda|,\qquad q_0=|\lambda|-(n-k)$,
	\item[(i):] $p_i=-|\lambda|+\lambda^*_{n-k+1-i}+(i-1),\qquad q_i=|\lambda|-(n-k)-(\lambda^*_{n-k+1-i}-1)$.
\end{enumerate}

	We immediately see that $p_i+q_i=-(n-k)+i$ are all distinct for $i=0,\ldots,n-k$. Combining with the fact that the spectral sequence converges
	to a single sheaf concentrated in degree $0$, we get a long exact sequence
\[
0\to\HH^{q_0}(\Sigma^{\lambda^\prime}\CU(1)\otimes\Sigma^{\mu_0}\CU^\perp)\otimes \Sigma^{\mu_0^*}\CU\to\ldots\to \HH^{q_{n-k}}(\Sigma^{\lambda^\prime}\CU(1)\otimes\Sigma^{\mu_0}\CU^\perp)\otimes \Sigma^{\mu_{n-k}^*}\CU\to \Sigma^{\lambda^\prime}\CU(1)\to 0.
\]

	Finally, let us compute $\HH^{q_i}(\Sigma^{\lambda^\prime}\CU(1)\otimes \Sigma^{\mu_i}\CU^\perp)$. Once we permute sequence~\eqref{bbwsq1} so that it
	becomes strictly decreasing, it will be of the form
\[
(n+1,n,\ldots,\widehat{n+1-\nu_i},\ldots,1).
\]
	Thus, applying the Borel--Bott--Weil theorem we get $\HH^{q_i}(\Sigma^{\lambda^\prime}\CU(1)\otimes \Sigma^{\mu_i}\CU^\perp)\simeq \Lambda^{\nu_i}V$. Finally, we note that $\mu_0=\lambda^*$ and $\nu_0=n$.
	As $\Lambda^n V\simeq\kk$, dualizing we get the desired long exact sequence. This finishes the proof.
\end{proof}

There is a nice combinatorial way to describe $\mu_i$ and $\nu_i$ from Proposition~\ref{les}. Given a diagram $\lambda\in\You_{n,k}$ with
$\lambda_1=n-k$, draw a stripe of width 1, as shown on Figure~\ref{fig:first}.

\begin{figure}
	\scalebox{1} 
	{
	\begin{pspicture}(0,-2.08)(6.74,2.1)
	\psframe[linewidth=0.02,dimen=outer](6.66,1.6)(0.66,-2.0)
	\psline[linewidth=0.04](6.66,1.6)(6.66,1.0)(4.86,1.0)(4.86,-0.2)(4.26,-0.2)(4.26,-0.8)(2.46,-0.8)(2.46,-1.4)(1.26,-1.4)(1.26,-2.0)(0.66,-2.0)
	\psline[linewidth=0.04,linestyle=dashed,dash=0.16cm 0.16cm](6.06,1.6)(4.26,1.6)(4.26,0.4)(3.66,0.4)(3.66,-0.2)(1.86,-0.2)(1.86,-0.8)(0.66,-0.8)(0.66,-1.4)(0.06,-1.4)(0.06,-2.0)
	\usefont{T1}{ptm}{m}{n}
	\rput(3.66,1.905){$n-k$}
	\usefont{T1}{ptm}{m}{n}
	\rput(0.37,-0.155){$k$}
	\usefont{T1}{ptm}{m}{n}
	\rput(5.22,0.225){$\lambda$}
	\usefont{T1}{ptm}{m}{n}
	\rput(2.84,0.125){$\lambda^\prime(-1)$}
	\psdots[dotsize=0.12](6.06,1.6)
	\psdots[dotsize=0.12](6.66,1.6)
	\psdots[dotsize=0.12](0.66,-2.0)
	\psdots[dotsize=0.12](0.06,-2.0)
	\end{pspicture} 
	}
	\caption{}\label{fig:first}
\end{figure}

Now, $\mu_i$ is pictured on Figure~\ref{fig:mui} by the solid line. One takes the path $\lambda$ going from left to right and ``jumps'' upward on the path
$\lambda^\prime(-1)$ in the point with abscissa $n-k-i$.

\begin{figure}[h]
	\scalebox{1} 
	{
	\begin{pspicture}(0,-2.41)(6.68,2.41)
	\definecolor{color132b}{rgb}{0.8,0.8,0.8}
	\pspolygon[linewidth=0.02,linecolor=white,fillstyle=solid,fillcolor=color132b](6.6,1.8)(6.6,1.2)(4.8,1.2)(4.8,0.0)(4.2,0.0)(4.2,-0.6)(3.6,-0.6)(3.6,0.6)(4.2,0.6)(4.2,1.8)
	\psframe[linewidth=0.02,dimen=outer](6.6,1.8)(0.6,-1.8)
	\usefont{T1}{ptm}{m}{n}
	\rput(2.09,2.065){$n-k-i$}
	\usefont{T1}{ptm}{m}{n}
	\rput(0.31,0.045){$k$}
	\psdots[dotsize=0.12](6.6,1.8)
	\psdots[dotsize=0.12](0.6,-1.8)
	\psline[linewidth=0.04,linestyle=dashed,dash=0.16cm 0.16cm](3.6,-0.6)(4.2,-0.6)(4.2,0.0)(4.8,0.0)(4.8,1.2)(6.6,1.2)(6.6,1.8)
	\psline[linewidth=0.04,linestyle=dashed,dash=0.16cm 0.16cm](3.6,0.0)(1.8,0.0)(1.8,-0.6)(0.6,-0.6)(0.6,-1.2)(0.0,-1.2)(0.0,-1.8)
	\usefont{T1}{ptm}{m}{n}
	\rput(2.72,-1.195){$\mu_i$}
	\psline[linewidth=0.04](6.6,1.8)(4.2,1.8)(4.2,0.6)(3.6,0.6)(3.6,0.0)(3.6,-0.6)(2.4,-0.6)(2.4,-1.2)(1.2,-1.2)(1.2,-1.8)(0.6,-1.8)
	\psline[linewidth=0.02,fillcolor=color132b,linestyle=dashed,dash=0.16cm 0.16cm](3.6,2.4)(3.6,-2.4)
	\usefont{T1}{ptm}{m}{n}
	\rput(5.09,2.065){$i$}
	\end{pspicture} 
	}
	\caption{}\label{fig:mui}
\end{figure}

The number $\nu_i$ is the number of boxes one needs to remove from $\lambda$ in order to get $\mu_i$. These a pictured in gray on Figure~\ref{fig:mui}.


\subsection{Fullness} 
\label{sub:fullness}

In this section we will finish the proof of Theorem~\ref{theorem:brav}. We have already shown semi-orthogonality in Proposition~\ref{prop:exc_b}. In order to
show fullness it is convenient to introduce another Lefschetz decomposition.

Consider the categories
\[
	\CBp_i = \left<\Sigma^\lambda\CU^*\mid \lambda\in\Youl_{n,k},\ i<l(\lambda)\right>.
\]

\begin{example}
	In the case $X=\Gr(3,6)$ we have
\[
	\begin{array}{l}
		\CBp_0=\CBp_1=\left<\CO_X(3),\ \Lambda^2\CU^*(2),\ \CU^*(2),\ \Sigma^{(3,3,1)}\CU^*,\ \Sigma^{(3,2,1)}\CU^*\right> \\
		\CBp_2=\CBp_3=\left<\CO_X(3),\ \Lambda^2\CU^*(2),\ \CU^*(2)\right> \\
		\CBp_4=\CBp_5=\left<\CO_X(3),\ \Lambda^2\CU^*(2)\right> \\
	\end{array}
\]
\end{example}

We will need the following observation.

\begin{lemma}\label{lemma:b_prime}
	There is a Lefschetz decomposition
\[
	\left<\CB_0,\CB_1(1),\ldots, \CB_{n-1}(n-1)\right>=\D^b(X)
\]
	if and only if there is a decomposition 
\begin{equation}\label{eq:bp}
	\left<\CBp_{n-1}(1-n),\ldots,\CBp_1(-1),\CB_0\right>=\D^b(X).
\end{equation}
\end{lemma}
\begin{proof}
	The duality functor is an anti-auto-equivalence of categories, and the twist functor by $\CO_X(n-k)$ is an auto-equivalence. Composing these functors one gets
	an anti-auto-equivalence that preserves fullness and inverts orthogonality relations. It is left to note that $\CBp_i = \CB_i^*(n-k)$.
\end{proof}

Due to the previous lemma it will be enough to show that the decomposition~\eqref{eq:bp} is full in order to prove fullness of the decomposition $\CB_i(i)$.

\begin{proposition}\label{prop:bp_full}
	The decomposition $\CBp_i(-i)$ is full. In other words,
\[
	\left<\CBp_{n-1}(1-n),\ldots,\CBp_1(-1),\CBp_0\right> = \D^b(X).
\]
\end{proposition}

\begin{proof}
	We recall that the definitions of $r(\lambda)$, $l(\lambda)$, $d(\lambda)$ and $e(\lambda)$ are given in the end of Section~\ref{sub:orders}.
	
	Given a diagram $\lambda$ such that $\lambda(t)\in\You_{n,k}$ for some $t$, let $\tilde{\lambda}$ denote the diagram $\lambda(-\lambda_1+n-k)$. It will be convenient to
	dualize Kapranov's collection and work with $\left(\Sigma^\lambda\CU^*\mid\lambda\in\You_{n,k}\right)$. Once again, the duality functor is an anti-auto-equivalence, thus, it is
	a full exceptional collection in $\D^b(X)$.
	
	It will be enough to show that every object from $\left(\Sigma^\lambda\CU^*\mid\lambda\in\You_{n,k}\right)$ has a resolution with all the terms being direct sums of vector bundles
	$\Sigma^\mu\CU^*(-t)$, where $\mu\in\Youl_{n,k}$ and $0\leq t<l(\mu)$.
	
	Given a vector bundle $\Sigma^\lambda\CU^*(-t)$ with $\lambda\in\You_{n,k}$, one can twist it by $\CO_X(t+n-k-\lambda_1)$ and get a vector bundle satisfying
	the conditions of Proposition~\ref{les}. Twisting back the long exact sequence~\eqref{eq:staircase}, one gets a resolution of the original vector bundle. Informally speaking,
	given a vector bundle $\Sigma^{\lambda}\CU^*$ with $\lambda\in\You_{n,k}$ we would like to take this resolution. Replace every term that is not of the form
	$\Sigma^\mu\CU^*(-t)$ for some $\mu\in\Youl_{n,k}$ with its resolution and repeat this process until it terminates. Now we can fold all this ``multicomplex'' into a giant
	resolution. However, we need to show that the process actually terminates and that the bundles $\Sigma^\mu\CU^*$ with $\mu\in\Youl_{n,k}$ arise with appropriate twists.
	Let us formalize this argument.
	
	For each $\lambda\in\You_{n,k}$, such that $\lambda_1=n-k$ and $\lambda\notin \Youl_{n,k}$, consider the set
\[
	\mathrm{Exp}(\lambda) = \left\{(\tilde{\mu}_1, t_1),\ \ldots,\  (\tilde{\mu}_{n-k}, t_{n-k}),\ (\tilde{\lambda^\prime},\ t_{n-k+1})\right\},
\]
	where $t_i=(n-k-\mu_{i,1})$ for $1\leq i\leq n-k$ and $t_{n-k+1}=(n-k-\lambda^\prime_1)+1$. Define
	$\mathrm{Exp}(\lambda)=\left\{(\lambda,0)\right\}$ for $\lambda\in\Youl_{n,k}$ and denote $\mathrm{Exp^{(1)}}(\lambda)=\mathrm{Exp}(\lambda)$. Further, define
\[
	\mathrm{Exp^{(k+1)}}(\lambda)=\bigcup_{(\mu, t)\in \mathrm{Exp^{(k)}}(\lambda)}\mathrm{Exp}(\mu)[t],
\]
	where $\left\{(\mu_i,t_i)\right\}_{i\in I}[t]=\left\{(\mu_i,t_i+t)\right\}_{i\in I}$.
	
	Note that
\[
	\Sigma^\lambda\CU^*\in\left<\Sigma^\mu\CU^*(-t)\right>_{(\mu,t)\in \mathrm{Exp}(\lambda)}
\]
	by Proposition~\ref{les}. By induction we also have
\[
	\Sigma^\lambda\CU^*\in\left<\Sigma^\mu\CU^*(-t)\right>_{(\mu,t)\in \mathrm{Exp^{(k)}}(\lambda)}
\]
	for all $k\geq 1$.
	
	We claim that
\begin{enumerate}
	\item for some $k\geq 1$ one has $\mathrm{Exp^{(k)}}(\lambda)\subset \Youl_{n,k}\times\ZZ$,
	\item if $\mathrm{Exp^{(k)}}(\lambda)\subset\Youl_{n,k}\times\ZZ$, then for each $(\mu,t)\in\mathrm{Exp^{(k)}}(\lambda)$ one has $t<l(\mu)$.
\end{enumerate}

	For (1) we note that
\[
	\max\{d(\mu)\}_{(\mu,t)\in\mathrm{Exp^{(k)}}}\leq \max\{d(\lambda)-1, 0\}.
\]
	It follows that for all $(\mu,t)\in\mathrm{Exp^{d(\lambda)}}(\lambda)$ we have $d(\mu)\leq 0$, hence $\mu\in\Youl_{n,k}$.
	
	For (2) we note that for each $(\mu,t)\in\mathrm{Exp}(\lambda)$ we have $t\leq e(\mu)-e(\lambda)$. By induction this also holds for each
	$(\mu,t)\in\mathrm{Exp^{(k)}}(\lambda)$. Now, if $\mathrm{Exp^{(k)}}(\lambda) \subset \Youl_{n,k}\times\ZZ$ then
\[
	t\leq e(\mu)-e(\lambda)=l(\mu)-e(\lambda)<l(\mu),
\]
	as by definition $e(\lambda)>0$.
	
	We have shown that $\Sigma^\lambda\CU^*\in \left<\CBp_{n-1}(1-n),\ldots,\CBp_1(-1),\CBp_0\right>$ for all $\lambda\in \You_{n,k}$, $\lambda_1=n-k$. For an arbitrary $\lambda\in\You_{n,k}$
	note that
\[
	\Sigma^\lambda\CU^*\in\left<\Sigma^\mu\CU^*(-t)\mid {(\mu,t)\in \mathrm{Exp^{(k)}}(\tilde{\lambda})[n-k-\lambda_1]}\right>.
\]
	
	We need to show that for all $(\mu,t)\in\mathrm{Exp^{d(\tilde{\lambda})}}(\tilde{\lambda})$ one has $t+(n-k-\lambda_1)<l(\mu)$. We already know that $t\leq l(\mu)-e(\tilde{\lambda})$. It is easy to see that
\[
	e(\tilde{\lambda}) = e(\lambda(n-k-\lambda_1)) = e(\lambda)+(n-k-\lambda_1) > (n-k-\lambda_1),
\]
	as $e(\lambda)>0$. Therefore,
\[
	t\leq l(\mu)-e(\tilde{\lambda})=l(\mu)-e(\lambda)-(n-k-\lambda_1) < l(\mu)-(n-k-\lambda_1).
\]
	This completes the proof.
\end{proof}

\begin{proof}[Proof of Theorem~\ref{theorem:brav}]
	Semi-orthogonality was proved in Proposition~\ref{prop:exc_b}. Combining Lemma~\ref{lemma:b_prime} and Proposition~\ref{prop:bp_full} we get fullness.
\end{proof}

Finally, let us return to our main example.

\begin{proposition}
	Conjecture~\ref{conj:full} holds for $X=\Gr(3,6)$.
\end{proposition}
\begin{proof}
	An argument similar to Proposition~\ref{lemma:b_prime} shows that it is enough to check that
\[
	\D^b(X)=\left<\Sigma^\lambda\CU^*(-t)\mid\lambda\in\Youml_{n,k},\ t< o(\lambda)\right>,
\]
	where $\Youml_{n,k}=\left\{\lambda\mid\lambda^c\in\Youmu_{n,k}\right\}$.
	
	In our case
\[
	\Youl_{6,3}\setminus\Youml_{6,3}=\{(3,3,1)\}.
\]

	We already know from Proposition~\ref{prop:bp_full} that
\[
	\D^b(X)=\left<\Sigma^\lambda\CU^*(-t)\mid\lambda\in\Youl_{n,k},\ t< l(\lambda)\right>.
\]
	Thus, it will be enough to show that
\begin{equation}\label{eq:322}
	\Sigma^{(3,3,1)}\CU^*, \Sigma^{(3,3,1)}\CU^*(-1)\in \left< \Sigma^\lambda\CU^*(-t)\mid\lambda\in\Youml_{n,k},\ t< o(\lambda)\right>.
\end{equation}

	Like in the proof of Proposition~\ref{prop:bp_full} consider the sets $\mathrm{Exp^{(k)}}$, but this time $\mathrm{Exp}(\lambda)=\{(\lambda, 0)\}$ only
	for $\lambda\in\Youml_{n,k}$. Then it is easy to check that $E=\mathrm{Exp^{(2)}}((3,3,1))\subset\Youml_{6,3}\times\ZZ$. However,
	the sets $E$ and $E(1)$ contain bad elements $((3,2,1),2)$ and $((3,2,1),3)$ for which $o((3,2,1))=2$.
	
	In order to show~\eqref{eq:322} it is enough to check that
\begin{equation}\label{eq:321}
	\Sigma^{(3,2,1)}\CU^*(-2),\Sigma^{(3,2,1)}\CU^*(-3)\in \left<\Sigma^\lambda\CU^*(-t)\mid\lambda\in\Youml_{n,k},\ t< o(\lambda)\right>.
\end{equation}
	
	Consider the long exact sequence~\eqref{eq:staircase} for $\Sigma^{(3,2,1)}\CU^*$
\begin{multline*}
	0\to\Sigma^{(3,2,1)}\CU^*(-2)\to\Lambda^5V^*\otimes\Sigma^{(3,2,2)}\CU^*(-2)\to \\
	\to \Lambda^3V^*\otimes\Sigma^{(3,3,3)}\CU^*(-2)\to V^*\otimes
	\Sigma^{(3,3,2)}\CU^*(-1)\to\Sigma^{(3,2,1)}\CU^*\to 0
\end{multline*}
	as a right resolution for $\Sigma^{(3,2,1)}\CU^*(-2)$. We immediately see that~\eqref{eq:321} holds.
\end{proof}




\def\cprime{$'$}



\end{document}